\documentclass{amsart}

\usepackage{amssymb,amsthm,amsfonts,latexsym,mathtools,thmtools}
\usepackage[utf8]{inputenc}
\usepackage{xcolor}
\usepackage{mathtools}
\usepackage{hyperref}
\usepackage{mathdots}
\usepackage{array}
\usepackage{enumitem}
\usepackage{bbm}
\usepackage{comment}
\usepackage{caption} 

\newcommand{\C}{\mathbb{C}}
\newcommand{\F}{\mathbb{F}}

\newcommand{\1}{\mathbbm{1}}

\newcommand{\calC}{\mathcal{C}}

\newcommand{\calL}{\mathcal{L}}
\newcommand{\frakg}{\mathfrak{g}}

\newcommand{\Sym}{\mathbb{S}}
\newcommand{\Alt}{\mathbb{A}}
\newcommand{\SL}{\mathrm{SL}}
\newcommand{\PSL}{\mathrm{PSL}}

\newcommand{\PSU}{\mathrm{PSU}}

\newcommand{\Sz}{\mathrm{Sz}}

\newcommand{\GL}{\mathrm{GL}}

\newcommand{\Tr}{\mathrm{Tr}}
\newcommand{\ad}{\mathrm{ad}}

\numberwithin{equation}{section}

\newtheorem{theorem}{Theorem}[section]
\newtheorem{lemma}[theorem]{Lemma}
\newtheorem{proposition}[theorem]{Proposition}

\newtheorem{remark}[theorem]{Remark}
\newtheorem{question}[theorem]{Question}
\newtheorem{conjecture}[theorem]{Conjecture}

\theoremstyle{definition}
\newtheorem{definition}[theorem]{Definition}

\title[Non-degeneracy of Killing forms]{Non-degeneracy of Killing forms on real conjugacy classes of finite groups}

\address{Department of Mathematics and Data Science, Vrije Universiteit Brussel, Pleinlaan 2, 1050 Brussel, Belgium}

\author{Carsten Dietzel}
\email{carsten.dietzel@vub.be}

\author{Charlotte Roelants}
\email{charlotte.roelants@vub.be}

\begin{document}

\begin{abstract}
    Killing forms on finite groups arise as special cases of braided Killing forms on braided Lie algebras. If $\calC$ is a conjugation-stable subset of a finite group $G$, the Killing form on $\C\calC$ is given by $K_\calC(a,b) = |C_G(ab) \cap \calC|$ for $a,b \in \calC$. 
    It is conjectured in previous work by López Peña, Majid and Rietsch that $K_\calC$ is non-degenerate for any real conjugacy class $\calC$ in a finite simple group.
    
    In this article, we reformulate the conjecture and introduce combinatorial conditions - the \emph{1-element condition} and the \emph{2-element condition} - that are sufficient for non-degeneracy to hold. This allows us to prove the conjecture for simple groups of the form $\PSL_2(q)$ and certain conjugacy classes in the alternating and symmetric groups. Moreover, we verify computationally that every real conjugacy class in a simple group of order $\leq 10^9$ fulfills at least one of these two conditions, thereby significantly extending the computational evidence for the conjecture. This raises the question whether these conditions are satisfied by all conjugacy classes in finite simple groups.
\end{abstract}

\maketitle

\section{Introduction}
\label{sec:introduction}

Killing forms were introduced in the context of Lie algebras: for a Lie algebra $\frakg$ over a field $k$ with Lie bracket $[ -, -]$, and for each $x \in \frakg$, $\ad_x: \frakg \to \frakg; z \mapsto [x,z]$ denotes the adjoint endomorphism of $x$ on $\frakg$.
The \textit{Killing form} is then the bilinear form given by
\[
K : \frakg \times \frakg \longrightarrow k, (x,y) \mapsto \Tr(\ad_x \circ \ad_y),
\]
with $\Tr$ denoting the trace map.

If this form is invertible, it is said to be \textit{non-degenerate}, and \textit{degenerate} otherwise.
The study of non-degeneracy of Killing forms is at least partially motivated by Cartan's criterion (see \cite{Cartan}), which states that a finite-dimensional Lie algebra over a field of characteristic zero is semisimple if and only if its associated Killing form is non-degenerate.

In order to generalize the Lie algebra of a Lie group to certain Hopf algebras, the notion of Lie algebra was extended in \cite{Majid_braided} to that of a braided Lie algebra $(\calL, \Delta, \epsilon, [-,-])$. The author also introduced braided Killing forms on braided Lie algebras. For $x,y \in \calL$, this is defined as the braided trace of the map $[x, [y, -] ] \in \mathrm{End}_k(\calL)$. We refer to \cite{Majid_braided} for more details.

A prominent family of braided Lie algebras can be constructed from subsets of finite groups, as proposed in \cite{LopezPenaMajidRietsch}. In these cases, the Killing forms can be described in purely group-theoretic terms and have a close relation to character theory. Motivated by Cartan's criterion, this construction thus provides a particularly interesting class of examples to investigate the non-degeneracy of braided Killing forms on braided Lie algebras.

Concretely, let $\calC \subset G \setminus \{1\}$ be a conjugation-stable subset of a finite group $G$ and consider $\calL = \C \calC$. Then $\calL$ can be viewed as an object in the braided monoidal category of $\C$-vector spaces with the trivial braiding operator and the unit object $\C$. There is a unique structure of a coalgebra on $\calL$ whose coproduct satisfies $\Delta(a) = a \otimes a$ and whose counit satisfies $\epsilon(a) = 1$, whenever $a \in \calC$.
The bilinear extension of the assignment $[a,b] = {}^ab := aba^{-1}$ ($a,b \in \calC$) then defines a bracket operation on $\calL$ such that $\calL$ satisfies the axioms of a braided Lie algebra (for more details, see \cite[Section 2.7]{beggs_majid} and \cite{LopezPenaMajidRietsch, Majid_braided}).

\begin{definition} \label{def:killing_form}
    Let $G$ be a finite group and $\calC \subset G \setminus\{1\}$ a subset stable under conjugation by elements of $G$. The \textit{Killing form} on $\calC$ is the braided Killing form of the braided Lie algebra $(\calL, \Delta, \epsilon, [-,-])$ as defined above. For $a,b \in \calC$, this is given by
    \[
    K_\calC(a,b) = \Tr_\calL( [a, [b, -] ]) = |\{x \in \calC : {}^{ab}x = x\}| = |C_G(ab) \cap \calC|,
    \]
    The Killing form is then the extension of $K_\calC$ to a bilinear map $\C\calC \times \C\calC \to \C$.
\end{definition}

\smallskip 

It is asked in \cite{LopezPenaMajidRietsch} when the Killing form on a conjugation-stable subset $\calC$ is non-degenerate. In particular, the authors pose the following conjecture on Killing forms of real conjugacy classes. Recall that a conjugacy class is said to be \textit{real} if it is stable under inversion.

\begin{conjecture} \label{conj:non-degeneracy}
    Let $\calC$ be a non-trivial real conjugacy class in a finite non-abelian simple group. Then the Killing form on $\calC$ is non-degenerate.
\end{conjecture}

They supported this conjecture with computer calculations for groups up to order 75,000. 
However, evidence for infinite families of simple groups, or even for specific families of conjugacy classes, is still quite limited.
In \cite{LopezPenaMajidRietsch}, it was proven that for the conjugacy class of transpositions in a symmetric group, the corresponding Killing form is non-degenerate. In \cite{PitermanRoelants}, non-degeneracy was also verified for the conjugacy classes of involutions in $\PSL_2(2^n), \ \PSU_3(2^n)$ and $\Sz(2^{2n+1})$ with $n \geq 1$, the groups of Lie type and Lie rank one in characteristic two.

\smallskip

The goal of this article is to develop new methods that allow us to extend the computational support for \autoref{conj:non-degeneracy}, and to prove the conjecture for the family of groups $\PSL_2(q)$. Furthermore, we apply our methods to certain conjugacy classes in $\Alt_n$ and $\Sym_n$.
We summarize the main results of this article below.

Our main tools to confirm non-degeneracy of $K_\calC$ for a real conjugacy class $\calC$ are the following two sufficient conditions, which we refer to as the \emph{1-element condition} and \emph{2-element condition} respectively.
\begin{align}
        & \text{There exist } x,z \in \calC \text{ such that there is a unique } y \in \calC \text{ for which } {}^y x = z. \label{eq:1EC_conj} \tag{1EC} \\
        &\text{There exist } x, y, a_0, a_1 \in \calC \text{ such that} \label{eq:2EC_conj} \tag{2EC} \\
        &\text{(1) For all } z \in \calC, \text{ we have the equivalence } {}^zx = y \Leftrightarrow z \in \{ a_0,a_1 \}. \notag \\
        &\text{(2) There exists } g \in G \text{ of odd order such that } ^ga_0 = a_1. \notag
\end{align}

Observe that the 2-element condition specializes to the 1-element condition in the case when $a_0 = a_1$, as then the identity $g = e$ has odd order and satisfies ${}^ga_0 = a_1$. 
However, we decided to formulate \eqref{eq:1EC_conj} as a distinct condition here because in plenty of cases, it is sufficient to check for \eqref{eq:1EC_conj}, which can be verified more efficiently than $\eqref{eq:2EC_conj}$, which is more technical.

\begin{theorem}
\label{prop:reduction_summary}
    Let $G$ be a finite group and $\calC$ a real conjugacy class of $G$.
    If condition \eqref{eq:1EC_conj} or \eqref{eq:2EC_conj} is fulfilled, then $K_\calC$ is non-degenerate.
\end{theorem}

Using these conditions, we provide the first proof of Conjecture \ref{conj:non-degeneracy} for an infinite family of finite simple groups, namely $\PSL_2(q)$.

\begin{theorem} \label{thm:PSL2_all_nondeg}
    Let $\calC$ be a real conjugacy class in $\PSL_2(q)$, with $q$ a prime power.
    Then the Killing form $K_\calC$ is non-degenerate.
\end{theorem}

Next, we investigate conjugacy classes of involutions in the symmetric groups $\Sym_n$ and the alternating groups $\Alt_n$, thus building on the case of transpositions in $\Sym_n$ that has been dealt with in \cite{LopezPenaMajidRietsch}.

\begin{theorem} \label{thm:symm_involutions_summary}
    Let $\calC$ be a conjugacy class of involutions in $\Sym_n$ or $\Alt_n$, with $n \geq 2$. 
    Then the Killing form $K_\calC$ is non-degenerate, unless $n=4$ and $\calC$ is the conjugacy class of $(1 \ 2)(3 \ 4)$.
\end{theorem}

Furthermore, we study conjugacy classes of cycles of odd length in the alternating and symmetric groups. Our main result concerning these cases is the following.

\begin{theorem} \label{thm:alt_cycles_summary}
    For $n \geq 5$, let $\calC$ be a real conjugacy class of cycles of odd length $m$ in $\Sym_n$ or $\Alt_n$, where $n \geq m \geq 3$.
    Then $K_{\calC}$ is non-degenerate.
\end{theorem}

Lastly, we verify computationally that \eqref{eq:1EC_conj} or \eqref{eq:2EC_conj} hold for all real conjugacy classes in simple groups of order up to $10^9$.

Although \eqref{eq:1EC_conj} and \eqref{eq:2EC_conj} are not necessary conditions for $K_\calC$ to be non-degenerate, we observe that at least one of them still holds true for every real conjugacy class in a simple group of order $\leq 10^9$. To the best of the authors’ knowledge, this combinatorial phenomenon in finite simple groups has not yet been addressed in the literature and might be of independent interest.

\smallskip

The paper is organized as follows.

In \autoref{sec:reformulation}, we reformulate the problem of checking non-degeneracy for a real conjugacy class $\calC$. We prove that for such a class, the Killing form $K_\calC$ is non-degenerate if and only if a certain set of permutation matrices, stemming from the conjugation representation of the group, is linearly independent. This is equivalent to a specific set of characteristic vectors generating the $\C$-vector space $\C\calC$. 
These reframings of the non-degeneracy problem are then applied to prove \autoref{prop:reduction_summary}.

Next, we investigate the groups $\PSL_2(q)$. In \autoref{sec:preliminaries}, we recall the description of conjugacy classes in the groups $\PSL_2(q)$. The proof of \autoref{thm:PSL2_all_nondeg} is then covered in Sections \ref{sec:semisimple_odd}, \ref{sec:unipotent_odd} and \ref{sec:semisimple_even}, as we consider the cases of semisimple elements for odd $q$, unipotent elements for odd $q$ and semisimple elements for even $q$ separately.

\autoref{sec:symm_involutions} addresses certain conjugacy classes in the symmetric and alternating groups. 
We study classes of involutions and classes of cycles of odd length and prove \autoref{thm:symm_involutions_summary} and \autoref{thm:alt_cycles_summary}. 

Finally, in \autoref{sec:computations}, we explain how the verification of \eqref{eq:1EC_conj} and \eqref{eq:2EC_conj} was implemented in \texttt{GAP} in order to confirm \autoref{conj:non-degeneracy} for non-abelian simple groups of order up to $10^9$. Based on our observations, we pose new questions related to the 1-element condition and the 2-element condition.

\subsection*{Acknowledgements}

We thank Kevin Piterman for a thorough reading and suggestions that helped to significantly improve the quality of the manuscript. Furthermore, we express our gratitude to Andrew Darlington and Edouard Feingesicht for fruitful discussions related to the research problem.

The first author is currently supported by the FWO Senior Postdoctoral Fellowship FWOTM1244.
The second author is currently supported by the FWO Senior Research Project G004124N.
This work was partially supported by the project OZR3762 of Vrije Universiteit Brussel.

\section{Reformulation of the problem}
\label{sec:reformulation}

For real conjugacy classes in a finite group $G$, we state a condition equivalent to non-degeneracy, which lies at the basis of \autoref{prop:reduction_summary}.

\smallskip

Suppose that $G$ is a finite group, and $\calC$ is a non-empty subset of $G \setminus \{ 1 \}$ that is stable under inversion and conjugation by $G$.
Define
\[
K'_\calC(a,b) = K_\calC(a,b^{-1}), \qquad a,b \in \calC.
\]
As $\calC$ is stable under inversion, the underlying matrix of $K'_\calC$ is obtained from $K_\calC$ by a re-ordering of the columns, therefore $K_\calC$ is non-degenerate if and only $K'_\calC$ is.

\smallskip

We now remark that $K_\calC$ can also be expressed in character theoretic terms as follows. Let $\rho: G \rightarrow \GL(\C\calC)$ be the conjugation representation of $G$ on $\calC$, where for $g \in G$,
\[
\rho_g: \C \calC \rightarrow \C\calC, \ \sum_{x \in \calC} \alpha_x x \mapsto \sum_{x \in \calC} \alpha_x \; ^g x.
\]
Let $\chi: G \rightarrow \C$ denote the character of $\rho$ and for $g \in G$, let $P_g$ be the matrix representation of $\rho_g$ with respect to the canonical basis of $\C\calC$. Then for $a,b \in \calC$,
\begin{equation} \label{eq:KC_character}
    K_\calC(a,b) = \Tr(P_{ab}) = \chi(ab).
\end{equation}
Note that for any $g \in G$, the representing matrix $P_g$ of $\rho_g$ is a permutation matrix. 
In particular, each $P_g$ is unitary, and thus $P_{g^{-1}} = P_g^{-1} = \overline{P}_g^T$.
Combining this with Equation \eqref{eq:KC_character}, we can rewrite $K'_\calC$ as
\[
K'_\calC(a,b) = \Tr(P_{ab^{-1}}) = \Tr(P_a P_b^{-1}) = \Tr(P_a \overline{P}_b^T).
\]

Let $V = M_n(\C)$ and consider the map 
\[
\mu: V \times V \longrightarrow \C, (A,B) \mapsto \Tr(A\overline{B}^T) = \sum_{i,j = 1}^n A_{ij} \overline{B}_{ij}.
\]
Note that for any $A \neq 0$, $\mu(A,A) > 0$. Thus $\mu$ is a positive definite, sesquilinear form and, in particular, non-degenerate.

\begin{lemma} \label{lem:reformulation_linearindep}
    Let $G$ and $\calC$ be as above. The following are equivalent:
    \begin{enumerate}
        \item $K_\calC$ is non-degenerate.
        \item $K'_\calC$ is non-degenerate.
        \item The set of permutation matrices $\{P_c : c \in \calC \}$ is linearly independent.
    \end{enumerate}
\end{lemma}

\begin{proof}
    The equivalence of the first two statements is discussed above.
    
    For the equivalence of (2) and (3), first note that $K'_\calC(a,b) = \mu(P_a, P_b)$.

    If $K'_\calC$ is degenerate, there exists $x = \sum_{a \in \calC} \lambda_a a \neq 0$ in $\C \calC$ such that $K'_\calC(x,y) = 0$ for all $y = \sum_{b \in \calC} \nu_b b \in \C\calC$.
    Thus
    \[
    0 = K'_\calC(x,y) = \sum_{a,b \in \calC} \lambda_a \nu_b K'_\calC(a,b) = \sum_{a,b \in \calC} \lambda_a \nu_b \mu(P_a, P_b) = \mu (\sum_{a \in \calC} \lambda_a P_a, \sum_{b \in \calC} \overline{\nu_b} P_b),
    \]
    for all $y = \sum_{b \in \calC} \nu_b b \in \C\calC$.
    Since $\mu$ is non-degenerate, this implies $\sum_{a \in \calC} \lambda_a P_a = 0$.

    On the other hand, suppose $\sum_{a \in \calC} \lambda_a P_a = 0$ with $\lambda_a \in \C$ not all equal to 0.
    Then for all $\sum_{b \in \calC} \nu_b b \in \C\calC$,
    \[
    0 = \mu \Big(\sum_{a \in \calC} \lambda_a P_a, \sum_{b \in \calC} \nu_b P_b \Big) = K'_\calC \Big(\sum_{a \in \calC} \lambda_a a, \sum_{b \in \calC} \overline{\nu_b} b \Big).
    \]
    Thus $K'_\calC$ is degenerate.
\end{proof}

\begin{remark}
    Recall that the adjoint representation of a finite group $G$ is given by $\ad : G \to \mathrm{End}_{\C}(\C G)$, $g \to \ad_g$, where $\ad_g(h) = {}^gh$ for $g,h \in G$. A group is said to have the \emph{Roth property} if the kernel of its adjoint representation is trivial, or equivalently, if the adjoint representation contains every irreducible representation of the group. The problem of which groups satisfy the Roth property was first studied in \cite{Passman,Roth} and solved for all finite simple groups in \cite{HSTZ,HeideZalesski}.

    Note that (3) can be checked to be equivalent to the condition that the kernel of the adjoint map $\ad: \C\calC \to \mathrm{End}_{\C}(\C\calC)$, $x \mapsto \ad_x = [x,-]$ is trivial and can therefore be seen as a variant of the Roth property. In the context of the non-degeneracy problem, it has been proven in \cite[Theorema 4.2]{LopezPenaMajidRietsch} that for a group $G$, the validity of the Roth property implies that $K_{\calC}$ is non-degenerate for $\calC = G \setminus \{ 1 \}$. We can thus view \autoref{lem:reformulation_linearindep} as a variation of this earlier result.
\end{remark}

In the following, we associate with a subset $S \subseteq \calC$ the element $\1_S = \sum_{x \in S} x \in \C\calC$ - the \emph{characteristic vector} of the subset $S$.

Suppose $\lambda_c$  ($c \in \calC$) are coefficients in $\C$ such that $\sum_{c \in \calC} \lambda_c P_c = 0$.
For $x,y \in \calC$, we define the set 
\[
\calC_{y,x} = \{c \in \calC : \; ^c x = y\} = \{c \in \calC : (P_c)_{yx} = 1\}, 
\]
where $(P_c)_{yx}$ denotes the entry on position $(y,x)$ in the matrix $P_c$.
Then, for each $x,y \in \calC$, 
\[
0 = \Big(\sum_{c \in \calC} \lambda_c P_c \Big)_{yx} = \sum_{c \in \calC} \lambda_c (P_c)_{yx} = \sum_{c \in \calC_{y,x}} \lambda_c.
\]
This can also be written as an inner product in $\C\calC$: for $\lambda = (\lambda_c)_{c \in \calC}$,
\[
0 = \sum_{c \in \calC_{y,x}} \lambda_c = \langle \lambda, \1_{\calC_{y,x}} \rangle.
\]
Together with \autoref{lem:reformulation_linearindep}, this proves the following lemma:

\begin{lemma} \label{lem:reformulation}
    Let $G$ and $\calC$ be as above. The following statements are equivalent:
    \begin{enumerate}
        \item $K_{\calC}$ is non-degenerate.
        \item The set $\{ P_c : c \in \calC\}$ is linearly independent.
        \item For $\lambda \in \C\calC$, if $\langle \lambda, \1_{\calC_{y,x}} \rangle = 0$ for all $x,y \in \calC$, then $\lambda = 0$.
        \item The set $\{\1_{\calC_{y,x}}: x,y \in \calC \}$ generates $\C\calC$ as a $\C$-module.
    \end{enumerate}
\end{lemma}

\smallskip

From now on, we take $\calC$ to be a non-trivial, real conjugacy class in a finite group $G$. In that case, we can rewrite conditions \eqref{eq:1EC_conj} and \eqref{eq:2EC_conj} using the sets $\calC_{y,x}$, as well as in terms of the centralizer $C_G(x)$ of an element $x$ in $G$.
\begin{align}
    &\text{There exist } x,y,z \in \calC \text{ such that } \calC_{z,x} = \{y\}. \tag{1EC} \\
    &\text{There exist } x,y \in \calC \text{ such that } yC_G(x) \cap \calC = \{y\}. \tag{1EC'} \label{eq:1EC_cent}
\end{align}

\begin{align} 
    &\text{There exist } x, y, a_0, a_1 \in \calC \text{ such that}  \tag{2EC}\\
    &\text{(1) } \calC_{y,x} = \{ a_0, a_1 \}. \notag \\
    &\text{(2) There exists } g \in G \text{ of odd order such that } ^ga_0 = a_1. \notag \\
    &\text{There exist } a_0, a_1, x \in \calC \text{ such that} \tag{2EC'} \label{eq:2EC_cent} \\
    &\text{(1) } a_0 C_G(x) \cap \calC = \{ a_0, a_1 \}. \notag \\
    &\text{(2) There exists } g \in G \text{ of odd order such that } ^ga_0 = a_1. \notag 
\end{align}

We prove that conditions \eqref{eq:1EC_conj} and \eqref{eq:1EC_cent} and conditions \eqref{eq:2EC_conj} and \eqref{eq:2EC_cent} respectively are equivalent to each other, and that each of them is a sufficient condition for a real conjugacy class to have a non-degenerate Killing form. In order to do so, we need the following lemma.

\begin{lemma} \label{lem:shifting_Cyx}
    Let $\calC$ be a conjugacy class in a group $G$, then for all $x,y \in \calC$ and $g \in G$, we have
    \[
    \calC_{{}^gy,{}^gx} = {}^g\calC_{y,x}.
    \]
\end{lemma}

\begin{proof}
    Let $z \in \calC_{y,x}$ and $g \in G$, then the element $h = {}^gz$ satisfies
    \[
    {}^h({}^gx) = {}^{hg}x = {}^{gzg^{-1}g}x = {}^g({}^zx) = {}^gy,
    \]
    so $h \in \calC_{{}^gy,{}^gx}$. It follows that ${}^g\calC_{y,x} \subseteq \calC_{{}^gy,{}^gx}$. Using this inclusion, we obtain
    \[
    \calC_{{}^gy,{}^gx} = {}^g({}^{g^{-1}} \calC_{{}^gy,{}^gx}) \subseteq {}^g \calC_{{}^{g^{-1}}({}^gy),{}^{g^{-1}}({}^gx)} = {}^g \calC_{y,x},
    \]
    thus proving equality.
\end{proof}

\begin{proposition}
\label{lm:reduction1el}
Let $G$ and $\calC$ be as above.
Conditions \eqref{eq:1EC_conj} and \eqref{eq:1EC_cent} are equivalent. 
Additionally, if $\calC$ satisfies \eqref{eq:1EC_conj} or \eqref{eq:1EC_cent}, then $K_{\calC}$ is non-degenerate.
\end{proposition}

\begin{proof}
    Note that for $x,y,a \in \calC$, ${}^y x = {}^a x$ is equivalent to $a^{-1}y x = x a^{-1} y$ and thus $a^{-1} y \in C_G(x)$. The equivalence of \eqref{eq:1EC_conj} and \eqref{eq:1EC_cent} then immediately follows.
    
    Now suppose \eqref{eq:1EC_conj} holds, so $y$ is the only element in $\calC$ with $(P_y)_{zx} \neq 0$. 
    Then, if $(\lambda_c)_{c \in \calC} \in \C \calC$ such that $\sum_{c \in \calC} \lambda_c P_c = 0$, this implies
    \[
    0 = \Big(\sum_{c \in \calC} \lambda_c P_c \Big)_{zx} = \sum_{c \in \calC} \lambda_c (P_c)_{zx} = \sum_{a \in \calC_{z,x}} \lambda_a = \lambda_y.
    \]
    
    Moreover, any other element $c \in \calC$ can be written as $^hy$ for some $h \in G$. Then \autoref{lem:shifting_Cyx} implies that $\calC_{{}^hz,{}^hx} = {}^h\calC_{z,x} = {}^h\{ y \} = \{ c \}$.
    From the reasoning above, it follows that $\lambda_c = 0$ for all $c \in \calC$.
    Thus $\{ P_c : c \in \calC\}$ is linearly independent, and by \autoref{lem:reformulation_linearindep}, $K_\calC$ is non-degenerate.
\end{proof}

\begin{proposition}
\label{lm:reduction2el}
Let $G$ and $\calC$ be as above.
Conditions \eqref{eq:2EC_conj} and \eqref{eq:2EC_cent} are equivalent.
In addition, if $\calC$ satisfies \eqref{eq:2EC_conj} or \eqref{eq:2EC_cent}, then $K_\calC$ is non-degenerate.
\end{proposition}

\begin{proof}
    The equivalence of \eqref{eq:2EC_conj} and \eqref{eq:2EC_cent} follows in the same way as in the proof of \autoref{lm:reduction1el}.
    
    Suppose \eqref{eq:2EC_conj} holds. Then there are $x,y \in \calC$ such that $\calC_{y,x} = \{a_0,a_1\}$ for some $a_0,a_1 \in \calC$, together with an element $g \in G$ of odd order $k$ such that $^ga_0 = a_1$. We can assume that $a_0 \neq a_1$ as the case $a_0 = a_1$ has already been dealt with in \autoref{lm:reduction1el}. Then by \autoref{lem:shifting_Cyx},
    \[
    \calC_{^gy, ^gx} = \, ^g \calC_{y,x} = \{ ^ga_0, \,^ga_1 \} = \{a_1, a_2\},
    \]
    where we set $a_2 = \, ^ga_1$. 

    Set $a_i = \, ^ga_{i-1} = \, ^{g^i} a_0$ for $0 \leq i < k$. Then continuing the reasoning above yields
    \[
    \calC_{^{g^i}y, \, ^{g^i}x} = \{a_i, a_{i+1}\}, \qquad \text{for } 0 \leq i < k.
    \]

    Remark that $a_k = \, ^{g^k} a_0 = a_0$. Hence
    \[
    \calC_{^{g^{k-1}}y, \, ^{g^{k-1}}x} = \{a_{k-1}, a_k\} = \{a_{k-1}, a_0 \}.
    \]

    Therefore,
    \begin{align*}
        \sum_{i=0}^{k-1} (-1)^i \1_{\calC_{^{g^i}y, ^{g^i}x}}
        &= \sum_{i=0}^{k-1} (-1)^i (\1_{\{a_i\}} + \1_{\{a_{i+1}\}}) \\
        &= \1_{\{a_0\}} + \1_{\{a_1\}} - \1_{\{a_1\}} - \1_{\{a_2\}} + \dots + \1_{\{a_{k-1}\}} + \1_{\{a_0\}} \\
        &= 2 \cdot \1_{\{a_0\}}.
    \end{align*}

    This shows that $\1_{\{a_0\}}$ can be written as a linear combination of characteristic vectors $\1_{\calC_{^{g^i}y, \, ^{g^i}x}}$, with $0 \leq i < k$.

    Similar as in the proof of \autoref{lm:reduction1el}, each $c \in \calC$ can be written as $^ha_0$ for some $h \in G$. Hence for $c \in \calC$, $\1_{\{c\}}$ can be written as a linear combination of characteristic vectors $\1_{\calC_{^h(^{g^i}y), \, ^h(^{g^i}x)}}$, with $0 \leq i < k$. 
    As the vectors $\{ \1_{\{c\}} : c \in \calC \}$ generate $\C\calC$ as a $\C$-module, the set $\{\1_{\calC_{y,x}}: x,y \in \calC \}$  is also generating. 
    By \autoref{lem:reformulation_linearindep} and \autoref{lem:reformulation}, $K_\calC$ is thus non-degenerate.
\end{proof}

We remark here that the above proof shows that the conclusion of \autoref{lm:reduction2el} still holds true if condition (2) is replaced by the more general assumption that the undirected graph on the vertex set $\calC$ whose edges are given by $\{ {}^ga_0, {}^ga_1 \}$ ($g \in G$) contains an odd cycle. However, condition (2) appears to be easier to check than this more general condition, so we decided to formulate \autoref{lm:reduction2el} in that special case.

\section{Preliminaries on \texorpdfstring{$\PSL_2(q)$}{PSL\_2(q)}}
\label{sec:preliminaries}

In the next sections, we will prove \autoref{thm:PSL2_all_nondeg}. We first recall some important facts and results about the conjugacy classes in this family of linear groups. For more details, see for example \cite[Section 2.8]{Gorenstein}, \cite[Sections 1.9 and 3.6]{Suzuki} and \cite[Section 3.3]{Wilson}.

\smallskip

Let $q$ be a power of a prime number $p$. 
Recall that $\SL_2(q)$ is the group of all $(2 \times 2)$-matrices $A$ over $\F_q$ whose determinant is $1$. Then
$\PSL_2(q) = \SL_2(q) / Z(\SL_2(q))$, where $Z(\SL_2(q)) = \{I_2\}$ if $q$ is even and $\{I_2, -I_2\}$ if $q$ is odd.
In the following sections, we will denote matrices in $\SL_2(q)$ with uppercase letters $A, B,$ etc. When $q$ is even, so $\SL_2(q) = \PSL_2(q)$, we use this notation for elements in $\PSL_2(q)$ as well. For $q$ odd, we denote the image of $A \in \SL_2(q)$ as $\Bar{A} \in \PSL_2(q)$.

\smallskip

The following are well-known facts about conjugacy classes in $\PSL_2(q)$. The classes in these groups are, for the most part, determined by the characteristic polynomials or the traces of the representative matrices in $\SL_2(q)$ of their elements. In particular, for 
\[
A = \begin{pmatrix}
    a_{11} & a_{12} \\
    a_{21} & a_{22}
\end{pmatrix} \in \SL_2(q),
\]
the characteristic polynomial of $A$ is $\chi_A(t) = t^2 - (a_{11} + a_{22}) t +1 = t^2 - \Tr(A) t +1$. 

The non-identity elements in $\PSL_2(q)$ can be split into \textit{unipotent} and \textit{semisimple} elements. The properties of unipotent elements are listed in Table \ref{tab:unipotent_classes}.

\begin{table}[h]
    \centering
    \begin{tabular}{|m{3.5cm}|m{4cm}|m{4cm}|}
    \hline
    $q$ even & $q \equiv 1 \pmod 4$ & $q \equiv 3 \pmod 4$ \\
    \hline \hline
    $\chi_A(t)$ has $1$ as its unique double root. & \multicolumn{2}{m{8cm}|}{$\chi_A(t)$ has either $1$ or $-1$ as its unique double root, depending on the chosen representative $A$.} \\
    \hline 
    $\Tr(A) = 0$ & \multicolumn{2}{m{8cm}|}{$\Tr(A) = \pm 2$, depending on the chosen representative $A$.} \\
    \hline
    One conjugacy class of unipotent elements. & \multicolumn{2}{m{8cm}|}{Two conjugacy classes of unipotent elements.} \\
    \hline
    \multicolumn{3}{|m{11.5cm}|}{\hspace{3cm} Unipotent elements have order $p$.} \\
    \hline
    The conjugacy class of unipotent elements is real. & The conjugacy classes of unipotent elements are both real. & The conjugacy classes of unipotent elements are each other's inverse. \\
     \hline
    \end{tabular}
    
    \caption{Properties of a unipotent element $\Bar{A} \in \PSL_2(q)$.}
    \label{tab:unipotent_classes}
\end{table}

Note that the methods to verify non-degeneracy of the Killing form described in \autoref{sec:reformulation} do not apply to the conjugacy classes of unipotent elements in the case that $q \equiv 3 \pmod 4$ as these classes are not real. Therefore we will not consider them in this article.

When $q \equiv 1 \pmod 4$, the following lemma describes how to verify in which conjugacy class a unipotent element lies:

\begin{lemma} \label{lem:unipotent_distinct}
    Let $q$ be a power of an odd prime $p$. Suppose $q \equiv 1 \pmod 4$ and let $\calC$ be the conjugacy class of unipotent elements in $\PSL_2(q)$ represented by $\Bar{X}$, where
    \[
    X = \pm \begin{pmatrix}
        1 & 1 \\ 0 & 1
    \end{pmatrix}.
    \]
    Then, a unipotent element $\Bar{A} \in \PSL_2(q)$, where $A = (a_{ij})_{i,j=1,2} \in \SL_2(q)$, lies in $\calC$ if and only if $a_{12}$ and $a_{21}$ are squares in $\F_q$.
\end{lemma}

\begin{proof}
    First note that as $q \equiv 1 \pmod 4$, $-1 \in \F_q^{\times 2}$. Thus, it does not matter which matrix $A$ or $-A$ in $\SL_2(q)$ we choose to represent a unipotent element $\Bar{A}$ in $\PSL_2(q)$, as this does not change whether $a_{12}$ and $a_{21}$ are squares. We may therefore choose $A$ such that $\Tr(A) = 2$.

    Then $\Bar{A} \in \calC$ if and only if there exists a matrix $B = (b_{ij})_{i,j=1,2} \in \SL_2(q)$ such that
    \begin{equation} \label{eq:squares_antidiagonal}
    A = B \begin{pmatrix}
        1 & 1 \\ 0 & 1
    \end{pmatrix} B^{-1} = \begin{pmatrix}
        1 - b_{11} b_{21} & b_{11}^2 \\
        - b_{21}^2 & 1 + b_{11} b_{21}
    \end{pmatrix}.
    \end{equation}
    Therefore, for every conjugate $\Bar{A}$ of $\Bar{X}$, $A$ has squares on the anti-diagonal.

    On the other hand, suppose $a_{12} = \alpha^2$ and $a_{21} = -\beta^2$ for some $\alpha, \beta \in \F_q$. Note that $\alpha$ and $\beta$ cannot be zero simultaneously as $\Tr(A)=2$ implies that $A = I_2$, which is not unipotent.
    
    Now $\det(A) = a_{11}a_{22} + \alpha^2 \beta^2 = 1$. 
    If $a_{11} = 0$, then $\alpha \beta = \epsilon$ for some $\epsilon \in \{1, -1\}$, and $a_{22} = 2$. Thus $a_{11} = 1 - \epsilon \alpha \beta$ and $a_{22} = 1 + \epsilon\alpha \beta$.
    If $a_{11} \neq 0$, we can write $a_{22} = \frac{1-\alpha^2 \beta^2}{a_{11}}$. Then $\Tr(A) = a_{11} + a_{22} = 2$ implies $a_{11} = 1 - \epsilon \alpha \beta$ and $a_{22} = 1 + \epsilon \alpha \beta$, for some $\epsilon \in \{1, -1\}$. 
    In all cases, $A$ is of the form as in Equation \eqref{eq:squares_antidiagonal} for a suitable matrix $B$ with $b_{11} = \epsilon \alpha$ and $b_{21} = \beta$ and thus $\Bar{A}$ is indeed conjugate to $\Bar{X}$ in $\PSL_2(q)$.
\end{proof}

Non-unipotent elements in $\PSL_2(q)$ have order prime to $p$ and are termed \textit{semi\-simple}.
In $\PSL_2(q)$, two semisimple elements are conjugate if and only if they can be represented by matrices in $\SL_2(q)$ with the same characteristic polynomial and thus the same trace. 
One can easily derive from this that all conjugacy classes of semisimple elements in $\PSL_2(q)$ are real.
Moreover, for $q$ odd, involutions are semisimple and are represented by matrices $A \in \SL_2(q)$ with $\chi_A(t) = t^2 + 1$ or thus $\Tr(A) = 0$.

We shall make a distinction between \textit{semisimple split} and \textit{semisimple non-split} elements. The properties of both types of elements are listed in \autoref{tab:semisimple}.

\begin{table}
    \centering
    \begin{tabular}{|m{5.5cm}|m{5.5cm}|}
    \hline
    Semisimple split case & Semisimple non-split case \\
    \hline \hline
    $\chi_A(t)$ has 2 distinct roots in $\F_q$. & $\chi_A(t)$ has no roots in $\F_q$. \\
    \hline
    $\Tr(A) = \xi + \xi^{-1}$ for some $\xi \in \F_q^\times$. & $\Tr(A) = \xi + \xi^q$ for some $\xi \in \F_{q^2}^\times \setminus \F_q^\times$ such that $\xi^{q+1} = 1$. \\
    \hline
    For $q$ odd, $\Tr(A)^2 - 4$ is a square in $\F_q^\times$. & For $q$ odd, $\Tr(A)^2 - 4$ is not a square in $\F_q^\times$. \\
    \hline 
    The order of $\Bar{A}$ divides $q-1$ if $q$ is even and $\frac{q-1}{2}$ if $q$ is odd. & The order of $\Bar{A}$ divides $q+1$ if $q$ is even and $\frac{q+1}{2}$ if $q$ is odd. \\
    \hline 
    \multicolumn{2}{|m{11cm}|}{\hspace{2.7cm} The conjugacy class of $\Bar{A}$ is real.} \\
    \hline
    \end{tabular}
    
    \caption{Properties of a semisimple element $\Bar{A} \in \PSL_2(q)$.}
    \label{tab:semisimple}
\end{table}

Finally, in later sections we will need the following lemma, which follows from \cite[Theorem 5.48]{LidlNiederreiter}:

\begin{lemma} \label{lem:x^2+1_square}
    Let $q$ be an odd-prime power and $N$ the number of elements $x \in \F_q$ such that $x^2+1$ is a square. Then
    \[
    N = \begin{cases}
        \frac{q+1}{2} &\text{if } q \equiv 1 \pmod 4, \\
        \frac{q-1}{2} &\text{if } q \equiv 3 \pmod 4.
    \end{cases}
    \]
\end{lemma}

\section{Semisimple elements in \texorpdfstring{$\PSL_2(q)$}{PSL2(q)}, \texorpdfstring{$q$}{q} odd}
\label{sec:semisimple_odd}

In this section, we show that the 1-element condition holds for all semisimple conjugacy classes in a simple group of the form $\PSL_2(q)$ for an odd $q$. Recall from \autoref{tab:semisimple} that all of these classes are real.

We separately address the case of the conjugacy class $\calC$ of (non-trivial) involutions in $G = \PSL_2(q)$.

\begin{theorem} \label{thm:involutions}
    Let $\calC$ be the conjugacy class of involutions in $\PSL_2(q)$ with $q$ odd. 
    Then $\calC$ satisfies \eqref{eq:1EC_cent}. In particular, the Killing form $K_\calC$ is non-degenerate.
\end{theorem}

\begin{proof}
    The involutions in $\PSL_2(q)$ with $q$ odd are represented by matrices $A \in \SL_2(q)$ with characteristic polynomial $\chi_A(t) = t^2 + 1$.

Let $A$ be the matrix
\[
A = \begin{pmatrix}
    0 & 1 \\
    -1 & 0
\end{pmatrix}.
\]

Recall that as $G = \PSL_2(q) = \SL_2(q)/\{\pm I_2\}$, $\Bar{B} \in C_G(\Bar{A})$ if and only if $BA = \pm AB$. Defining $B_{x,y} = \begin{pmatrix}
    x & y \\
    -y & x
\end{pmatrix}$ and $B'_{x,y} = \begin{pmatrix}
    x & y \\
    y & -x
\end{pmatrix}$, we therefore obtain
\[
C_G(\Bar{A}) = \{ \Bar{B}_{x,y} : x,y \in \F_q, x^2 + y^2 = 1 \} \cup  \{ \Bar{B}'_{x,y} : x,y \in \F_q, x^2 + y^2 = - 1 \}.
\]

Considering the characteristic polynomial corresponding to the involutions, the class $\calC$ contains, in particular, the elements $\Bar{A}_{\alpha}$ with $\alpha \in \F_q$, where
\[
A_{\alpha} = \begin{pmatrix}
    \alpha & - 1 - \alpha^2 \\
    1 & - \alpha
\end{pmatrix}.
\]

Our goal is to find an element $\Bar{A}_{\alpha}$ such that $\Bar{A}_{\alpha} C_G(\Bar{A}) \cap \calC = \{\Bar{A}_{\alpha}\}$. In other words, we want to find $\alpha \in \F_q$ such that $\Tr(A_{\alpha} B) = 0$ implies $B = \pm I_2$ for $\Bar{B} \in C_G(\Bar{A})$.

First considering the case $B = B_{x,y}$, this implies
\[
\Tr(A_{\alpha} B_{x,y}) = (\alpha^2 +2) y.
\]
Thus $\Bar{A}_{\alpha} \Bar{B}_{x,y} \in \calC$ if and only if $y = 0$, except when $\alpha^2 = -2$. Thus if we choose $\alpha$ such that $\alpha^2 \neq -2$, the constraint $x^2 + y^2 = 1$ implies $x = \pm 1$ and thus $B_{x,y} = \pm I_2$.

On the other hand, in the case $B = B'_{x,y}$, we obtain
\[
\Tr(A_{\alpha} B'_{x,y}) = 2 \alpha x - \alpha^2 y.
\]
We now additionally assume that $\alpha \neq 0$. From the previous equation it then follows that $\Bar{A}_{\alpha} \Bar{B}'_{x,y} \in \calC$ if and only if this trace is equal to $0$, or equivalently,
\[
x = \frac{\alpha y}{2}.
\]
The constraint $x^2 + y^2 = - 1$ for the elements $B'_{x,y}$ then implies that
\begin{equation} \label{eq:involutions_tildezxy}
    \left( 1 + \frac{\alpha^2}{4}  \right)y^2 = -1.
\end{equation}

This shows that $B'_{x,y} \neq \pm I_2$ for all $x,y \in \F_q$. Therefore, we want to force Equation \eqref{eq:involutions_tildezxy} to not have any solutions, so that if $\Tr(A_{\alpha} B) = 0$ for any $B$ such that $\Bar{B} \in C_G(\Bar{A})$, the solution $B = \pm I_2$ that we obtained above is the only possibility. 
Thus, in order to ensure that Equation \eqref{eq:involutions_tildezxy} cannot hold, we will choose $\alpha$ in such a way that the left-hand side is a square if and only if $-1$ is not.

We multiply Equation \eqref{eq:involutions_tildezxy} by $4$. Note that then the left-hand side of Equation \eqref{eq:involutions_tildezxy} is a square if and only if $\alpha^2 + 4$ is. Furthermore, recall that we have obtained the additional constraint $\alpha^2 \neq -2$.

For $q \equiv 3 \pmod 4$, $-1$ is not a square in $\F_q^\times$ so we need to find $\alpha \in \F_q$ such that $\alpha^2 + 4$ is a square and $\alpha^2 \neq -2$.
By \autoref{lem:x^2+1_square}, there are $\frac{q-1}{2}$ elements $\beta$ such that $\beta^2 + 4$ is a square. One of these elements is $0$, and at most two of these elements satisfy $\beta^2 = -2$. Hence there are still at least $\frac{q-7}{2}$ choices for $\alpha$ with $\alpha \neq 0$. This implies that an appropriate choice for $\alpha$ exists when $q > 7$. When $q = 7$, it is readily checked that $\alpha = 2$ can be chosen.

Now suppose $q \equiv 1 \pmod 4$. In this case, $-1$ is a square so we need to find $\alpha \in \F_q$ such that $\alpha^2 + 4$ is not a square and $\alpha^2 \neq -2$.
By \autoref{lem:x^2+1_square}, there are $q - \frac{q+1}{2} = \frac{q-1}{2}$ elements $\beta \in \F_q$ such that $\beta^2 + 4$ is not a square. Again, at most two of these are a root of $-2$, and thus there are at least $\frac{q-5}{2}$ choices for $\alpha$. 
This lower bound only reaches 0 for $q=5$, but in this case we can clearly choose $\alpha = 2$ to satisfy our requirements.

For $\alpha$ chosen as above, Equation \eqref{eq:involutions_tildezxy} has no solutions.
Thus $B$ is forced to be $\pm I_2$ and therefore $\Bar{A}_{\alpha} C_G(\Bar{A}) \cap \calC = \{\Bar{A}_{\alpha}\}$. Thus, \eqref{eq:1EC_cent} is satisfied and by an application of \autoref{lm:reduction1el}, we conclude that $K_\calC$ is non-degenerate. \qedhere

\end{proof}

We now settle the case of non-involutive semisimple elements.

\begin{theorem}
    Let $\calC$ be a conjugacy class of non-involutive semisimple elements in $G = \PSL_2(q)$ with $q$ odd. Then $\calC$ satisfies \eqref{eq:1EC_cent}. In particular, $K_\calC$ is non-degenerate.
\end{theorem}

\begin{proof}

Let $\calC$ be as in the statement of the theorem. 
Recall from \autoref{sec:preliminaries} that all elements in $\calC$ can be represented by matrices in $\SL_2(q)$ with the same characteristic polynomial $\chi(t) = t^2 - bt + 1$, for a fixed $b \in \F_q$. Note that $b \notin \{0, \pm 2\}$ as the elements of $\calC$ are not unipotent or involutions.
Let $\delta = \left( \frac{b}{2} \right)^2 -1$. Then by Table \ref{tab:semisimple}, $\Bar{A}$ is split if and only if $\delta$ is a square in $\F_q$. Observe that $\chi_{-A}(t) = t^2 + bt + 1$, so $\delta$ is independent of the choice of a representative.

Fixing one choice of $b$, we now let $A$ be the matrix
\[
A = \begin{pmatrix}
    \frac{b}{2} & \delta \\ 1 & \frac{b}{2}
\end{pmatrix},
\]
and obtain an element $\Bar{A} \in \calC$. We then have
\[
C_G(\Bar{A}) = \left\{ \overline{\begin{pmatrix}
    x & y\delta \\ y & x
\end{pmatrix}} : x,y \in \F_q, x^2-\delta y^2 = 1\right\}.
\]

Furthermore, $\calC$ contains the elements $\Bar{A}_{\alpha,\gamma}$ ($\alpha, \gamma \in \F_q, \gamma \neq 0$) where
\[
A_{\alpha,\gamma} = \begin{pmatrix}
    \frac{b}{2}-\alpha & \frac{\delta - \alpha^2}{\gamma} \\ \gamma & \frac{b}{2}+\alpha
\end{pmatrix}.
\]
We want to find an element $\Bar{A}_{\alpha,\gamma}$ such that $\Bar{A}_{\alpha,\gamma}C_G(\Bar{A}) \cap \calC = \{ \Bar{A}_{\alpha,\gamma} \}$. This is the same as saying that $\Tr (A_{\alpha,\gamma}B) = b$, where $\Bar{B} \in C_G(\Bar{A})$, implies that $B = I_2$ or $B = -I_2$.

For an element $B \in \SL_2(q)$, we have $\Bar{B} \in C_G(\Bar{A})$ if and only if $B = B_{x,y} = \begin{pmatrix}
    x & y\delta \\ y & x
\end{pmatrix}$. With this, we calculate
\[
\Tr(A_{\alpha,\gamma}B_{x,y}) = bx + \left( \gamma\delta + \frac{\delta - \alpha^2}{\gamma} \right) y 
\]
Therefore, $\Bar{A}_{\alpha,\gamma}\Bar{B} \in \calC$ if and only if $\Bar{B}$ is represented by an element of the form $B_{x,y}$ such that $\Tr(A_{\alpha,\gamma}B_{x,y}) = b$, or equivalently,
\[
x = 1 - \left( \gamma\delta + \frac{\delta - \alpha^2}{\gamma} \right) \frac{y}{b}.
\]
Plugging this into the constraint $x^2 - \delta y^2 = 1$, we obtain a quadratic equation of the form
\[
c y^2 - 2 \left( \gamma\delta + \frac{\delta - \alpha^2}{\gamma} \right) \frac{y}{b} = 0,
\]
where 
\[
c = \Big( \frac{(\gamma^2 + 1) \delta - \alpha^2}{\gamma b} \Big)^2 - \delta. 
\]
This equation has the obvious solution $y = 0$, implying $x = 1$ and thus corresponding to the element $\Bar{B} = \Bar{B}_{1,0} = \Bar{I}_2 \in C_G(\Bar{A})$. This solution is unique if and only if the linear coefficient vanishes while the quadratic does not, that is:
\[
\frac{2}{b}\left( \gamma\delta + \frac{\delta - \alpha^2}{\gamma} \right) = 0 \quad \text{and} \quad c \neq 0.
\]
The first equation is equivalent to
\[
\delta(\gamma^2 +1) = \alpha^2.
\]
This equation has a solution if and only if either $\delta$ and $\gamma^2 +1$ are both squares or they are both non-squares.
Furthermore, in that case, 
\[
c = - \delta = 1 - \Big( \frac{b}{2} \Big)^2 \neq 0,
\]
as $\Bar{A}$ is not unipotent so $b \neq \pm 2$.

Whether $\delta \in \F_q^2$ depends on whether $\Bar{A}$ is split or non-split. In the split case, by \autoref{lem:x^2+1_square}, there is at least one choice for $\gamma \in \F_q$ such that $\gamma^2 + 1 \in \F_q^2$, and in the non-split case, there is a $\gamma \in \F_q$ such that $\gamma^2+1 \not\in \F_q^2$.
Hence in both cases there is a choice of an $\alpha, \gamma \in \F_q$ such that $\delta(\gamma^2 +1) = \alpha^2$. With this choice, $B$ is forced to be $I_2$ and therefore, $\Bar{A}_{\alpha,\gamma} C_G(\Bar{A}) \cap \calC = \{  \Bar{A}_{\alpha,\gamma} \}$. Hence, \eqref{eq:1EC_cent} is satisfied and $K_\calC$ is non-degenerate.
\end{proof}

\section{Unipotent elements in \texorpdfstring{$\PSL_2(q), \; q \equiv 1 \pmod 4$}{PSL\_2(q), q = 1 (mod 4)}} 
\label{sec:unipotent_odd}

The goal of this section is to prove that $K_\calC$ is non-degenerate when $\calC$ is a real conjugacy class of unipotent elements in $G = \PSL_2(q)$ with $q \equiv 1 \pmod 4$. We do so by verifying \eqref{eq:2EC_cent} for these classes.
Note that by Table \ref{tab:unipotent_classes}, for $q \equiv 3 \pmod 4$ the conjugacy classes of unipotent elements are not real and thus cannot be addressed with the methods in this paper. 
For $q$ even, the unipotent elements are exactly the involutions. This case has already been tackled in \cite{PitermanRoelants}.

In the case that $q \equiv 1 \pmod 4$, $\PSL_2(q)$ has two real conjugacy classes of unipotent elements. As these can be mapped into one another by the application of a suitable automorphism of $\PSL_2(q)$, it suffices to study the non-degeneracy of the Killing form on one of these classes.

\begin{theorem}
    Let $\calC$ be a real conjugacy class of unipotent elements in $\PSL_2(q)$ with $q \equiv 1 \pmod 4$. Then $\calC$ satisfies \eqref{eq:2EC_cent}. In particular, $K_\calC$ is non-degenerate.
\end{theorem}

\begin{proof}
    The unipotent elements in $G = \PSL_2(q)$ with $q$ odd are represented by matrices $A \in \SL_2(q)$ with characteristic polynomial $\chi_A(t) = t^2 \pm 2t + 1$. 
    
    We may choose 
    \[
    A = \begin{pmatrix}
        1 & 1 \\
        0 & 1
    \end{pmatrix}.
    \]
    Then clearly
    \[
    C_G(\Bar{A}) = \bigg\{ \overline{\begin{pmatrix}
        \epsilon & x \\
        0 & \epsilon
    \end{pmatrix}} : x \in \F_q, \epsilon \in \{ \pm 1 \} \bigg\}. 
    \]

    Furthermore, by \autoref{lem:unipotent_distinct}, the conjugacy class $\calC$ contains, in particular, the element
    \[
    A_0 = \begin{pmatrix}
        1 & 0 \\
        1 & 1
    \end{pmatrix}.
    \]

    Our goal is to check the 2-element condition in order to apply \autoref{lm:reduction2el}.
    Hence we first show that  $\Bar{A}_0 C_G(\Bar{A}) \cap \calC = \{ \Bar{A}_0, \Bar{A}_1 \}$ for some $\Bar{A}_1 \in \calC$.

    For $B_{x, \epsilon} = \begin{pmatrix}
        \epsilon & x \\
        0 & \epsilon
    \end{pmatrix}$, we find $\Tr(A_0 B_{x, \epsilon}) = 2 \epsilon + x$.

    Therefore, $\Bar{A}_0 \Bar{B} \in \calC$ only if there is a choice of $x, \epsilon$ with $\Bar{B} = \Bar{B}_{x,\epsilon}$ such that this trace is equal to $2$. For $\epsilon = 1$, this implies $x=0$. So in this case, we get $B = B_{0,1} = I_2$.

    If $\epsilon = -1$, we get $x = 4$. Thus, we have
    \[
    \Bar{A}_0 C_G(\Bar{A}) \cap \calC = \{ \Bar{A}_0, \Bar{A}_0 \Bar{B}_{4, -1} \}.
    \]

    Hence,
    \[
    A_1 = A_0 B_{4, -1} = \begin{pmatrix}
        1 & 0 \\
        1 & 1
    \end{pmatrix} \begin{pmatrix}
        -1 & 4 \\
        0 & -1
    \end{pmatrix} = \begin{pmatrix}
        -1 & 4 \\
        -1 & 3
    \end{pmatrix}.
    \]

    Note that, as $q \equiv 1 \pmod 4$, the elements on the anti-diagonals of $A_0$ and $A_1$ are squares so indeed, by \autoref{lem:unipotent_distinct}, $\Bar{A}_0, \Bar{A}_1 \in \calC$.
    
    We now check that conditions (1) and (2) of \eqref{eq:2EC_cent} are satisfied for $x = \Bar{A}$, $a_0 = \Bar{A}_0$ and $a_1 = \Bar{A}_1$.

    Condition (1) is fulfilled by construction.
    For condition (2), one can verify that
    \[
    J = \begin{pmatrix}
        -i & 2i \\
        0 & i
    \end{pmatrix},
    \]
    where $i \in \F_q$ is a root of $-1$, is an element such that ${}^{\Bar{J}}\Bar{A}_0 = \Bar{A}_1$, but $\Bar{J}$ is an involution in $\PSL_2(q)$ and thus does not satisfy condition (2).
    However, it does imply that all elements conjugating $\Bar{A}_0$ to $\Bar{A}_1$ are contained in $\Bar{J} C_G(\Bar{A}_0)$.
    Note that this centralizer contains all elements represented by lower-triangular matrices with ones on the diagonal, that is, 
    \[
    \bigg\{ \overline{\begin{pmatrix}
        1 & 0 \\
        c & 1
    \end{pmatrix}} : c \in \F_q \bigg\} \subset C_G(\Bar{A}_0).
    \]
    For a matrix $B_c$ of the form
    \[
    B_c = \begin{pmatrix}
        1 & 0 \\
        c & 1
    \end{pmatrix} \ (c \in \F_q),
    \]
    we have
    \[
    J B_c 
    = \begin{pmatrix}
        -i & 2i \\
        0 & i
    \end{pmatrix} \begin{pmatrix}
        1 & 0 \\
        c & 1
    \end{pmatrix}
    = \begin{pmatrix}
        i + 2ic & 2i \\
        -ic & -i
    \end{pmatrix}.
    \]
    The characteristic polynomial of this matrix is $\chi_{J B_c}(t) = t^2 - 2ic t +1$. Therefore, for $c = i$, $\chi_{J B_i}(t) = (t + 1)^2$ which shows that $g = \Bar{J} \Bar{B}_i$ is a unipotent element in $\Bar{J} C_G(\Bar{A}_0)$ and, in particular, of odd order. As ${}^g\Bar{A}_0 = \Bar{A}_1$, condition (2) of \autoref{lm:reduction2el} is satisfied.
\end{proof}

\section{Semisimple elements in \texorpdfstring{$\PSL_2(q), q$}{PSL\_2(q), q} even}
\label{sec:semisimple_even}

In this section, we prove the non-degeneracy of the Killing form on a conjugacy class of semisimple elements in $\PSL_2(q) = \SL_2(q)$, for $q$ even. We start with the semisimple split elements.

\begin{theorem}
    Let $\calC$ be a conjugacy class of semisimple split elements in $G = \PSL_2(q)$ with $q$ even. Then $\calC$ satisfies \eqref{eq:1EC_cent}. In particular, $K_\calC$ is non-degenerate.
\end{theorem}

\begin{proof}
Recall from Table \ref{tab:semisimple} that if $q$ is even, then a conjugacy class of semisimple split elements in $\PSL_2(q) = \SL_2(q)$ is represented by all matrices $A$ with characteristic polynomial $\chi_A(t) = t^2 + bt + 1$, for a fixed $b \in \F_q^\times$ such that $b = \xi + \xi^{-1}$ for some $\xi \in \F_q$. 

We choose
\[
A = \begin{pmatrix}
    \xi & 0 \\
    0 & \xi^{-1}
\end{pmatrix} \in \calC.
\]
Then the centralizer of $A$ is of the form
\[
C_G(A) = \Bigg\{ \begin{pmatrix}
    \lambda & 0 \\
    0 & \lambda^{-1}
\end{pmatrix} : \lambda \in \F_q^\times \Bigg\}.
\]
We choose the element
\[
A_0 = \begin{pmatrix}
    b & 1 \\
    1 & 0
\end{pmatrix} \in \calC.
\]

In order to invoke the 1-element condition, we show for $B \in C_G(A)$ that if $A_0 B \in \calC$, or equivalently $\Tr(A_0 B) = b$, then $B = I_2$. 

Each element in $C_G(A)$ is of the form $B = B_{\lambda} = \begin{pmatrix}
    \lambda & 0 \\
    0 & \lambda^{-1}
\end{pmatrix}$, where $\lambda \neq 0$. Furthermore, we have $\Tr(A_0 B_{\lambda}) = \lambda b$. Hence, $A_0 B_{\lambda} \in \calC$ if and only if this trace is equal to $b$ which implies $\lambda = 1$ and thus, $B = B_1 = I_2$. Therefore  $A_0 C_G(A) \cap \calC = \{A_0\}$.

As \eqref{eq:1EC_cent} is satisfied, \autoref{lm:reduction2el} implies that $K_\calC$ is non-degenerate.
\end{proof}

Next, we study the case of conjugacy classes of semisimple non-split elements.

\begin{theorem}
    Let $\calC$ be a conjugacy class of semisimple non-split elements in $G = \PSL_2(q)$ with $q$ even. Then $\calC$ satisfies \eqref{eq:2EC_cent}. In particular, $K_\calC$ is non-degenerate.
\end{theorem}

\begin{proof}
    Recall from \autoref{tab:semisimple} that a conjugacy class of semisimple non-split elements in $\PSL_2(q) = \SL_2(q)$ with $q$ even is represented by all matrices $A$ with characteristic polynomial $\chi_A(t) = t^2 + bt + 1$, for a fixed $b \in \F_q^\times$ that can be written as $b = \xi + \xi^q$ for some $\xi \in \F_{q^2}^\times \setminus \F_q^\times$, where $\xi^{q+1}=1$. 
    Hence we may take
    \[
    A = \begin{pmatrix}
        b & 1 \\
        1 & 0
    \end{pmatrix} \in \calC.
    \]

    Then the centralizer of $A$ is of the form
    \[
    C_G(A) = \bigg\{ \begin{pmatrix}
        bx + y & x \\
        x & y
    \end{pmatrix} : x,y \in \F_q, x^2 + y^2 + bxy = 1 \bigg\}.
    \]

    Furthermore, the conjugacy class $\calC$ contains, in particular, the matrices of the form $A_{\alpha}$ ($\alpha \in \F_q$), where
    \[
    A_{\alpha} = \begin{pmatrix}
        \alpha & 1 + (b + \alpha) \alpha \\
        1 & b + \alpha
    \end{pmatrix}.
    \]

    Our goal is to show that $\calC$ satisfies the 2-element condition. We do so by proving first that for any $\alpha \neq 0$, we have $ A_{\alpha} C_G(A) \cap \calC = \{ A_{\alpha}, A'\}$ for some $A' \in \calC$.

    Each element in $C_G(A)$ is of the form $B = B_{x,y} = \begin{pmatrix}
        bx + y & x \\
        x & y
    \end{pmatrix}$. With this, we get
    \[
    \Tr( A_{\alpha} B_{x,y}) = b y + \alpha^2 x.
    \]
    Thus $A_{\alpha} B_{x,y} \in \calC$ if and only if this trace is equal to $b$ which is the case if and only if $x = \frac{y+1}{\alpha^2} b$. Plugging this into the equation $x^2 + y^2 + bxy = 1$, we get
    \[
    \frac{y^2 + 1}{\alpha^4} b^2 + y^2 + \frac{y+1}{\alpha^2} y b^2 = 1,
    \]
    or equivalently,
    \[
    \Big( \frac{b^2}{\alpha^4} + \frac{b^2}{\alpha^2} + 1 \Big) y^2 + \frac{b^2}{\alpha^2} y + \frac{b^2}{\alpha^4} + 1 = 0.
    \]
    As $b \neq 0$, the linear coefficient of this equation does not vanish. Therefore, being of degree at most $2$, it has at most two solutions.
    One of them is $y=1$, which implies $x=0$ and thus $B = B_{0,1} = I_2$. As there is at most one other solution, we have proven that for $\alpha \neq 0$, we have $A_{\alpha} C_G(A) \cap \calC = \{ A_{\alpha}, A'\}$ for some $A'$, not necessarily different from $A_{\alpha}$.

    If $A_{\alpha} = A'$, clearly \eqref{eq:1EC_cent} is satisfied and the non-degeneracy of $K_{\calC}$ follows directly from \autoref{lm:reduction1el}.

    Suppose from now on that $\alpha \not\in \{0,b \}$ and that $A' \neq A_{\alpha}$. We prove first that $A' \neq A_{\alpha}^{-1}$. Assuming otherwise, we would obtain $A_{\alpha} \in A_{\alpha}^{-1}C_G(A)$ and thus, $A_{\alpha}^2 \in C_G(A)$. As $A_{\alpha}$ is \emph{not} an involution in $\PSL_2(q)$, it follows that $A_{\alpha}$ has odd order, therefore $A_{\alpha} \in \langle A_{\alpha}^2 \rangle \leq C_G(A)$. However, this would imply $A_{\alpha} = B_{x,y}$ for some $x,y \in \F_q$, and a comparison of the anti-diagonal entries quickly shows that this is only the case if $\alpha(\alpha + b) = 0$. This implies $\alpha \in \{0, b\}$, a contradiction! Therefore, $A' \neq A_{\alpha}^{-1}$.

    We now check that the requirements for \eqref{eq:2EC_cent} are satisfied for $x = A$, $a_0 = A_{\alpha}$ and $a_1 = A_1 = A'$. 
    
    Condition (1) is fulfilled by construction. 
    We prove that we can find an element $g$ of odd order as in condition (2) of \autoref{lm:reduction2el}. 
    Take any element $h \in G$ such that ${}^hA_0 = A_1$. If the order of $h$ is odd, we can take $g=h$.
    If the order of $h$ is even, then $h$ must be an involution, as these are the only elements of even order in $\PSL_2(q)$ when $q$ is even.
    Then ${}^{hA_0} A_0 = {}^h A_0 = A_1$.
    
    Suppose the order of $hA_0$ is also even, so $hA_0$ is an involution as well. 
    Then, as $h = h^{-1}$, we obtain $1 = (h A_0)^2 = {}^h A_0 A_0 = A_1 A_0$, and thus $A_1 = A_0^{-1}$, a contradiction. 
    Therefore, the order of $hA_0$ must be odd and we may take $g = hA_0$ in order to satisfy condition (2) in \autoref{lm:reduction2el}. It then follows that $K_\calC$ is non-degenerate.
\end{proof}

The results in Sections \ref{sec:semisimple_odd}, \ref{sec:unipotent_odd} and \ref{sec:semisimple_even}, as well as the results for unipotent elements in $\PSL_2(q)$ with $q$ even in \cite{PitermanRoelants}, prove \autoref{thm:PSL2_all_nondeg}.

\section{Classes in \texorpdfstring{$\Sym_n$}{Sn} and \texorpdfstring{$\Alt_n$}{An}} \label{sec:symm_involutions}

In this section, we will investigate the non-degeneracy of Killing forms for certain real conjugacy classes in the symmetric and the alternating groups.

For $n \geq 1$, we use $[n]$ to denote the set $\{1, 2, \dots, n\}$.

Throughout this section, we will make repeated use of the following, well-known lemma:

\begin{lemma} \label{lem:conjugacy_in_alternating_groups}
    Let $n \geq 2$ and let $x \in \Alt_n$. Then the conjugacy class of $x$ in $\Sym_n$ coincides with the conjugacy class of $x$ in $\Alt_n$ if and only if the cycle decomposition of $x$ does not consist of odd cycles of pairwise different sizes, $1$-cycles included.

    If the cycle decomposition of $x$ only involves odd cycles of pairwise different sizes, then the conjugacy class of $x$ in $\Sym_n$ is a union of two different conjugacy classes in $\Alt_n$. Moreover, for $g \in \Sym_n$, the elements $x$ and ${}^gx$ are conjugate in $\Alt_n$ if and only if $g \in \Alt_n$.
\end{lemma}

\begin{proof}
    See \cite[Theorem 11.1.5]{Scott} for the proof of the first two statements, whereas the last statement is an immediate consequence, due to the fact that $\Alt_n$ has index $2$ in $\Sym_n$.
\end{proof}

In this section, we prove \autoref{thm:symm_involutions_summary} by verifying the 1-element condition in the form of \eqref{eq:1EC_conj} in all relevant cases. Note that in the case of a class of transpositions, this result has already been proven in \cite{LopezPenaMajidRietsch}. 
On the other hand, in \cite{PitermanRoelants}, it is shown that for the conjugacy class of involutions with exactly one fixed point, the matrix associated to the Killing form can be decomposed as a diagonal block matrix, where the blocks are formed by grouping together elements that have the same unique fixed point. However, despite this diagonal block decomposition, the question of whether the matrix is invertible was left unanswered. 

\begin{theorem} \label{thm:sym_involutions}
    Suppose $n \geq 3$ and let $\calC$ be the conjugacy class of the involution $x = (1 \ 2) (3 \ 4) \dots (2k-1 \ 2k)$ in $\Sym_{n}$ or $\Alt_{n}$, for $1 \leq k \leq \lfloor \frac{n}{2} \rfloor$.
    Except when $n=4$ and $k = 2$, $\calC$ satisfies \eqref{eq:1EC_conj}. In particular, the Killing form $K_\calC$ is non-degenerate.
\end{theorem}

\begin{proof}
    The conclusion for the alternating groups follows easily from that for the symmetric groups since the $\Alt_{2n}$-conjugacy class $\calC$ of an involution coincides with its $\Sym_{2n}$-conjugacy class by \autoref{lem:conjugacy_in_alternating_groups}.

    Throughout this proof, equality is to be understood modulo $n$, that is, for all $i \in [n]$, we consider $i+n$ to be equal to $i$.

    \smallskip

    First suppose that either $2k \neq n$ or $k$ is odd. 
    Let $y = (2 \ 3) (4 \ 5) \dots (2k \ 2k+1) \in \calC$. Remark that if $n = 2k$, then $2k+1 = 1$.
    Thus for $i \in [2k]$ and $j \in \{2, \dots, 2k+1\}$,
    \begin{align}
        x(i) &= \begin{cases}
            i+1 &\text{ if $i$ is odd}, \\
            i-1 &\text{ if $i$ is even},
        \end{cases} \label{eq:symmetric_involution_x} \\
        y(j) &= \begin{cases}
            j-1 &\text{ if $j$ is odd}, \\
            j+1 &\text{ if $j$ is even}.
        \end{cases} \label{eq:symmetric_involution_y}
    \end{align}
    
    We prove that there exists a unique $z \in \calC$ such that ${}^zx = y$, or equivalently,
    \[
    \big(z(1) \ z(2) \big) \big( z(3) \ z(4) \big) \dots \big( z(2k-1) \ z(2k) \big) = (2 \ 3)(4 \ 5) \dots (2k \ 2k+1).
    \]
    Equations \eqref{eq:symmetric_involution_x} and \eqref{eq:symmetric_involution_y} then imply that for $i \in [2k]$,
    \begin{equation} \label{eq:symmetric_involution_z}
        z(i) = \begin{cases}
            z(i+1) - 1 &\text{ if $i$ is odd and $z(i)$ is even}, \\
            z(i+1) + 1 &\text{ if $i$ is odd and $z(i)$ is odd}, \\
            z(i-1) - 1 &\text{ if $i$ is even and $z(i)$ is even}, \\
            z(i-1) + 1 &\text{ if $i$ is even and $z(i)$ is odd}.
        \end{cases}
    \end{equation}

    Set $m = z(1)$.
    Note that, if $2k < n$, $m$ must be $2k+1$: if possible, $2k+1 < m  \leq n$ implies that $y$ has a 2-cycle containing $z(1) = m$, a contradiction as $y$ only moves the points $\{2, \dots, 2k+1\}$.
    On the other hand, if $m < 2k+1$, we have $z(m) = 1$ as $z$ is an involution. Because $m < 2k+1$, this forces $y$ to have a 2-cycle containing $z(m) = 1$, again a contradiction. Hence $m=2k+1$ is the only remaining possibility.
    
    Going back to the general case, that is, when $2k<n$ or $k$ is odd, we now claim that if $m$ is odd, then $z$ is the involution consisting of the transpositions $(j \ m+1-j)$ for all $1 \leq j \leq k$. If $m$ is even, we claim $z$ consists of the transpositions $(j \ m+j-1)$ for all $1 \leq j \leq k$. We prove this claim by induction on $j$. 
    Clearly the statement holds for $j = 1$. 
    
    Suppose then that $m$ is odd and $z$ contains the transposition $(j-1 \ m+2-j)$. 
    If $j-1$ is odd, then so is $m+2-j$. By Equation \eqref{eq:symmetric_involution_z}, 
    \[
    z(j-1) = z(j) + 1 \Longrightarrow z(j) = z(j-1) - 1 = m+1-j.
    \]
    If $j-1$ is even, then so is $m+2-j$ and thus again by Equation \eqref{eq:symmetric_involution_z},
    \[
    z(m+2-j) = z(m+1-j) - 1 \Longrightarrow z(m+1-j) = z(m+2-j) + 1 = j.
    \]
    Note that, if $2k < n$, we already have $m$ fixed to be $2k+1$ and so for $1 \leq j \leq k$, $m+1-j \in \{k+2, \dots, 2k+1\}$. 

    Now assume $m$ is even and $z$ contains the transposition $(j-1 \ m+j-2)$. This case immediately implies that $2k=n$, as even $m$ does not occur for $2k < n$.
    If $j-1$ is odd, then $m-2+j$ is even. Hence by Equation \eqref{eq:symmetric_involution_z},
    \[
    z(j-1) = z(j) - 1 \Longrightarrow z(j) = z(j-1) + 1 = m+j-1.
    \]
    If $j-1$ is even, then $m-2+j$ is odd and so
    \[
    z(m-2+j) = z(m-1+j) - 1 \Longrightarrow z(m-1+j) = z(m-2+j) + 1 = j.
    \]
    This concludes the proof of our claim.

    \smallskip

    We first consider the case $2k < n$. As mentioned before, in this case $m=2k+1$ and so, by our claim, $z$ is uniquely determined to be of the form
    \[
    z = (1 \ 2k+1) (2 \ 2k) \dots (k \ k+2).
    \]

    Next, we consider the case $n=2k$ and $k$ is odd. Remark that then $x,y$ and $z$ do not have any fixed points.
    If $m$ is odd, then by our previous claim, for $j = \frac{m+1}{2}$, $z$ contains the cycle $\Big( \frac{m+1}{2} \ m + 1 - \frac{m+1}{2} \Big) = \Big( \frac{m+1}{2} \frac{m+1}{2} \Big)$. Hence $\frac{m+1}{2}$ is a fixed point of $z$, a contradiction.
    
    We conclude that $m$ must be even. 
    For $j = m-1$, we now have that $z(m-1) = 2m-2$. Furthermore, as $m$ is even and $z(m) = 1$, $z(m) = z(m-1) + 1$ and thus $z(m-1) = z(m)-1 = n = 2k$, as we consider the equality modulo $n$. Thus
    \[
    2m-2 = 2k \Longrightarrow m = 1 \text{ or } m = k+1.
    \]
    If $m=1$, then $1$ is a fixed point of $z$, a contradiction as $z$ does not have any fixed points. Therefore, $m = k+1$ remains as the only possibility, and by our claim, this completely determines $z$ to be of the form
    \[
    z = (1 \ k+1) (2 \ k+2) \dots (k \ 2k).
    \]

    Finally, we consider the case $n = 2k$ and $k>2$ is even. 
    Let
    \[
    y = (1 \ n-2)(2 \ 3) \dots (n-4 \ n-3)(n-1 \ n) \in \calC .
    \]
    
    We prove that there is a unique $z \in \calC$ such that ${}^zx = y$, or equivalently, 
    \begin{multline*}
        \big(z(1) \ z(2) \big) \big( z(3) \ z(4) \big) \dots \big(z(n-3) \ z(n-2) \big) \big( (z(n-1) \ z(n) \big) \\
        = (1 \ n-2)(2 \ 3) \dots (n-4 \ n-3) (n-1\ n).
    \end{multline*}

    Suppose $z(n) = m$ with $m \in [n-2]$. As $z$ is an involution, this implies $z(m) = n$. 
    Hence by Equation \ref{eq:symmetric_involution_x}, ${}^zx$ contains the cycle
    \[
    \begin{cases}
        \big( z(m-1) \ z(m) \big) = \big( z(m-1) \ n \big) &\text{if $m$ is even}, \\
        \big( z(m) \ z(m+1) \big) = \big( n \ z(m+1) \big) &\text{if $m$ is odd}.
    \end{cases}
    \]
    Thus $z(n-1) = m-1$ if $m$ is even and $z(n-1) = m+1$ if $m$ is odd. 
    But then ${}^zx = y$ also must contain the cycle
    \[
    \big( z(n-1) \ z(n) \big) = \begin{cases}
        (m-1 \ m) &\text{if $m$ is even}, \\
        (m \ m+1) &\text{if $m$ is odd},
    \end{cases}
    \]
    which is a contradiction with Equation \eqref{eq:symmetric_involution_y}.

    Since $z$ has no fixed points, $z(n) \neq n$ which forces $z(n) = n-1$ and thus $(n-1 \ n)$ is a cycle contained in $x,y$ and $z$. 

    Let $x' = (n-1 \ n) x$, $y' = (n-1 \ n) y$ and $z' = (n-1 \ n) z$. We can then view $x', y'$ and $z'$ as involutions with no fixed points in $\Sym_{n-2}$. We denote their conjugacy class in $\Sym_{n-2}$ by $\calC'$. 
    By the case of $n=2k$ with $k$ odd, there is a unique choice for $z' \in \calC'$ such that ${}^{z'}x' = y'$. 
    Therefore, $z$ is unique in $\calC$ such that ${}^zx = y$.

    In all cases, we conclude that $\calC$ satisfies \eqref{eq:1EC_conj} and thus, by \autoref{lm:reduction1el}, $K_\calC$ is non-degenerate.
\end{proof}

\begin{remark}
    The case $n = 4$ and $k=2$ is excluded in \autoref{thm:sym_involutions}, as then in the proof, we cannot reduce to $\Sym_{n-2}$. In fact, one can easily verify that for $\Sym_4$ and $\Alt_4$, the conjugacy class of involutions with no fixed points has a degenerate Killing form. Of course, these groups are not simple and fall outside of the scope of \autoref{conj:non-degeneracy}.
\end{remark}

In the remainder of this section, we focus on conjugacy classes of cycles of odd length in the symmetric and the alternating groups. The results can be summarized as follows.

\begin{theorem} \label{theorem:nondegeneracy_odd_cycles}
    Let $m,n$ be integers with $n \geq 5$, $m = 2k+1$ for some $k \geq 1$ and $n \geq m$. Let $\calC \subseteq \Alt_n$ be the conjugacy class of the permutation $x = (1\ 2\ \ldots \ m)$. If $\calC$ is a real conjugacy class, $K_{\calC}$ is non-degenerate.
\end{theorem}

We begin by determining the values of $m$ and $n$ which have to be considered in \autoref{theorem:nondegeneracy_odd_cycles}.

\begin{proposition} \label{pro:determining_real_conjugacy_classes_an}
    Let $m,n$ be integers with $n \geq 5$, $m = 2k+1$ for some $k \geq 1$ and $n \geq m$. Let $\calC \subseteq \Alt_n$ be the conjugacy class of the permutation $x = (1\ 2\ \ldots \ m)$. Then $\calC$ is a real conjugacy class if and only if $n \geq m+2$ or $k$ is even.
\end{proposition}

Note that that the condition on $k$ in the second case is equivalent to the congruence $m \equiv 1 \pmod 4$.

\begin{proof}
    We need to check when $x$ and $x^{-1}$ are conjugate in $\Alt_n$. If $n \geq m+2$, then the cycle decomposition of $x$ involves $2$ cycles of length $1$, so according to \autoref{lem:conjugacy_in_alternating_groups}, $x$ and $x^{-1}$ are conjugate in $\Alt_n$.

    Now suppose that $m \leq n \leq m+1$, then the cycle decomposition involves only a cycle of odd length $m$ and, if $n = m+1$, a cycle of length $1$. Letting $g = (1\ m)(2\ m-1)\ldots (k \ k+2)$ where $m = 2k+1$, we have
    \[
    {}^gx = {}^g (1\ 2\ \ldots \ m) = (m\ m-1\ \ldots \ 1) = x^{-1}.
    \]
    Again using \autoref{lem:conjugacy_in_alternating_groups}, we see that $x$ and $x^{-1}$ are conjugate in $\Alt_n$ if and only if $g \in \Alt_n$. As $\mathrm{sgn}(g) = (-1)^k$, this is the case if and only if $k$ is even.
\end{proof}

We check \autoref{theorem:nondegeneracy_odd_cycles} for the cases with $m \in \{3,5,7\}$ separately:

\begin{proposition} \label{pro:3cycles_satisfy_2EC}
    Let $n \geq 5$ and let $\calC$ be the conjugacy class of $x = (1 \ 2 \ 3) \in \Alt_n$. Then $\calC$ is a real conjugacy class satisfying \eqref{eq:2EC_conj}. In particular, $K_{\calC}$ is non-degenerate.
\end{proposition}

\begin{proof}
    Let $z = (4 \ 3 \ 2)$. We determine $\calC_{z,x}$: let $y \in \calC$ be such that ${}^yx = z$. Then $y$ is determined on the subset $[3] \subseteq [n]$ by $y(1) \in \{2,3,4\}$. If $y(1) = 4$, then $y$ is forced to contain the $2$-cycle $(2\ 3)$ and can therefore not lie in $\calC$.

    If $y(1) = 2$, then $y(2) = 4$ and $y(3) = 3$, which forces $y = (1 \ 2 \ 4)$. Similarly, $y(1) = 3$ forces $y = (1 \ 3 \ 4)$. Therefore, $\calC_{z,x} = \{(1 \ 2 \ 4), (1 \ 3 \ 4)\}$. As ${}^g (1 \ 2 \ 4) = (1 \ 3 \ 4)$ for the element $g = (2 \ 3 \ 5)\in \Alt_n$ which is of odd order $3$, we have confirmed \eqref{eq:2EC_conj} for $\calC$. 
\end{proof}

\begin{proposition} \label{pro:5_cycles_satisfy_2_element_condition}
    Let $n \geq 5$ and let $\calC$ be the conjugacy class of $x = (1 \ 2 \ 3 \ 4 \ 5) \in \Alt_n$. Then $\calC$ is a real conjugacy class satisfying \eqref{eq:2EC_cent}. In particular, $K_{\calC}$ is non-degenerate.
\end{proposition}

\begin{proof}
    Note that $C_G(x) = \left\langle x \right\rangle \times \Alt_{n-5}$, where we identify $\Alt_{n-5}$ with the alternating group on the set $[n] \setminus [5]$.

    Let $y = (1 \ 3 \ 4 \ 2 \ 5) = {}^{(2 \ 3 \ 4)}x$. We determine $yC_G(x) \cap \calC$. To this end, let $g \in C_G(x)$. If $g$ contains the cycle $(1 \ 2 \ 3 \ 4 \ 5)^{\gamma}$ for $\gamma \in \{2,3,4\}$, then it is readily checked that $yg$ contains a non-trivial cycle of length $< 5$ and therefore, $yg \not\in \calC$. Therefore, $g$ either acts trivially on $[5]$ or contains the $5$-cycle $(1 \ 2 \ 3 \ 4 \ 5)$. In this latter case, $yg$ contains the $5$-cycle $(1 \ 4 \ 3 \ 5 \ 2)$. In order for $yg$ to be in $\calC$, $g$ has to act trivially on $[n] \setminus [5]$. Therefore, $yC_G(x) \cap \calC = \{(1 \ 3 \ 4 \ 2 \ 5), (1 \ 4 \ 3 \ 5 \ 2) \}$. As ${}^h(1 \ 3 \ 4 \ 2 \ 5) = (1 \ 4 \ 3 \ 5 \ 2)$ with $h = (1 \ 4 \ 5) \in \Alt_n$ which is of order $3$, we have confirmed \eqref{eq:2EC_cent} for $\calC$.
\end{proof}

\begin{proposition}\label{pro:7_cycles_satisfy_1_element_condition}
    Let $n \geq 9$ and let $\calC$ be the conjugacy class of $x = (1 \ 2\ \ldots \ 7) \in \Alt_n$. Then $\calC$ is a real conjugacy class satisfying \eqref{eq:1EC_cent}. In particular, $K_{\calC}$ is non-degenerate.
\end{proposition}

\begin{proof}
    Note first that $C_G(x) = \left\langle x \right\rangle \times \Alt_{n-7}$.

    Let $y = (1 \ 7 \ 2 \ 6 \ 3 \ 4 \ 5)$. We determine $yC_G(x) \cap \calC$. To this end, let $g \in C_G(x)$. The same argument as in the proof of \autoref{pro:5_cycles_satisfy_2_element_condition} shows that $yg \in \calC$ for $g \in C_G(x)$ implies that $g$ acts trivially on $[n] \setminus [7]$. Now if $g$ contains the cycle $x^{\gamma}$ for $1 \leq \gamma \leq 6$, one can check that $yg$ has a fixed point or a $2$-cycle on $[7]$ and therefore, $yg \not\in C_G(x)$. It follows that $g$ acts trivially on $[7]$ and thus, that $yC_G(x) \cap \calC = \{y\}$. We conclude that $\calC$ fulfills \eqref{eq:1EC_cent}.
\end{proof}

Before continuing with the proof of \autoref{theorem:nondegeneracy_odd_cycles} in the remaining cases, we need the following lemma.

\begin{lemma} \label{lem:conjugating_n_cycles}
    Let $n \geq 1$, $x = (1 \ 2 \ \ldots \ n) \in \Sym_n$ and let $y \in \Sym_n$ be an $n$-cycle. Define $g: [n] \to [n]$ by $g(i) = y^{i-1}(1)$ ($1 \leq i \leq n$). Then $g \in \Sym_n$ and $y = {}^gx$.
\end{lemma}

\begin{proof}
    As $y$ is an $n$-cycle, we have $y^i(1) = y^j(1)$ if and only if $j-i$ is divisible by $n$. Therefore $g$ is surjective and, as a consequence, a bijection.

    We now prove that $y = {}^gx$ which is equivalent to $gx = yg$. By definition of $g$, we have
    \[
    (gx)(i) = g(i+1) = y^i(1) = y(y^{i-1}(1)) = (yg)(i)
    \]
    for $1 \leq i < n$, while
    \[
    (gx)(n) = g(1) = 1 = y^n(1) = y(y^{n-1}(1)) = (yg)(n). \qedhere 
    \]
\end{proof}

We will now define a family of permutations $\pi_m$ that will be used to prove \autoref{theorem:nondegeneracy_odd_cycles} in the remaining cases. Let $m = 2k+1$ for some integer $k \geq 3$. We then define the permutation
\[
\pi_m: [m] \to [m] \quad ; \quad i \mapsto \begin{cases}
    m+1-i & 1 \leq i \leq \frac{m-5}{2}, \\
    \frac{m-1}{2} & i = \frac{m-3}{2}, \\
    \frac{m+3}{2} & i = \frac{m-1}{2}, \\
    1 & i = \frac{m+1}{2}, \\
    \frac{m+5}{2} & i = \frac{m+3}{2}, \\
    \frac{m+1}{2} & i = \frac{m+5}{2}, \\
    m+2-i & \frac{m+7}{2} \leq i \leq m.
\end{cases} 
\]

A standard induction shows that for $0 \leq i \leq m-1$ we have
\begin{equation} \label{eq:cycle_pi_m}
\pi_m^i(1) = \begin{cases}
    k+1 & i = 2k; \, k \leq \frac{m-5}{2}, \\
    m-k & i = 2k+1; \, k \leq \frac{m-7}{2}, \\
    \frac{m-1}{2} & i = m-4, \\
    \frac{m+3}{2} & i = m-3, \\
    \frac{m+5}{2} & i = m-2, \\
    \frac{m+1}{2} & i = m-1. \\
\end{cases}
\end{equation}

So $\pi_m$ is indeed an $m$-cycle.

\begin{lemma}
    Suppose that $n \geq m = 2k+1$ for some integer $k \geq 3$. Let $x = (1 \ 2 \ldots \ m)$ and
    \begin{equation} \label{eq:cycle_y}
    y: [n] \to [n] ; \quad i \mapsto \begin{cases}
        \pi_m(i) & 1 \leq i \leq m \\
        i & m+1 \leq i \leq n.
    \end{cases}
    \end{equation}
    Suppose the conjugacy class of $x$ in $\Alt_n$ is real. Then $x$ and $y$ are conjugate in $\Alt_n$.
\end{lemma}

\begin{proof}
    Due to \autoref{pro:determining_real_conjugacy_classes_an}, we distinguish between the cases $n \geq m+2$ and $m \equiv 1 \pmod 4$.

    For $n \geq m+2$, there is nothing to show as in that case, all $m$-cycles are conjugate in $\Alt_n$ by \autoref{lem:conjugacy_in_alternating_groups}.

    Now suppose that $m \equiv 1 \pmod 4$. Due to the obvious inclusion $\Alt_m \leq \Alt_n$, we may assume that $m = n$. Defining $g: [m] \to [m] $ by
    \[
    g(i) = \pi_m^{i-1}(1),
    \]
    it follows from \autoref{lem:conjugating_n_cycles} that $g \in \Sym_m$ and ${}^gx = y$. For the argument that is about to follow, it is convenient to represent $g$ as
    \[
    g = \begin{pmatrix}
        1 & 2 & 3 & 4 & \ldots & m-4 & m-3 & m-2 & m-1 & m \\
        1 & m & 2 & m-1 & \ldots & \frac{m-3}{2} & \frac{m-1}{2} & \frac{m+3}{2} & \frac{m+5}{2} & \frac{m+1}{2}
    \end{pmatrix}
    \]
    
    We now prove that $g \in \Alt_m$. Recall that $\mathrm{sgn}(g) = (-1)^{\ell (g)}$ where $\ell (g) = |\{ 1 \leq i < j \leq m : g(i) > g(j) \}|$. 

    Letting $D(j) = |\{ i \in [j-1] : g(i) > g(j) \}|$ ($1 \leq j \leq m$), we calculate
    \begin{align*}
        \ell(g) & = \sum_{j=1}^m D(j) \\
        & =  \sum_{k = 0}^{\frac{m-5}{2}} D(2k+1) 
        + \sum_{k = 0}^{\frac{m-7}{2}} D(2k+2)  
        +  \sum_{j = m-3}^{m-1} D(j) + D(m)  \\ 
        & =\sum_{k = 0}^{\frac{m-5}{2}} k+ \sum_{k = 0}^{\frac{m-7}{2}} k  + 3 \cdot \frac{m-5}{2} +  \frac{m-5}{2}+2  \\
        & = 2 \cdot  \sum_{k = 0}^{\frac{m-7}{2}} k  + \frac{m-5}{2} + 4 \cdot \frac{m-5}{2} + 2 \\
        & \equiv \frac{m-5}{2} \pmod 2.
    \end{align*}
    We now see that $\mathrm{sgn}(g) = (-1)^{\frac{m-5}{2}}$ which is equal to $1$ as $m \equiv 1 \pmod 4$.
\end{proof}

\begin{proposition} \label{pro:odd_cycles_satisfy_1EC}
    Let $n \geq m = 2k+1$ for some integer $k \geq 4$. Let $\calC$ be the conjugacy class of $x = (1 \ 2 \ldots \ m)$ in $G = \Alt_n$. If $\calC$ is a real conjugacy class, then $\calC$ satisfies \eqref{eq:1EC_cent}. In particular, $K_\calC$ is non-degenerate.
\end{proposition}

\begin{proof}
    Note first that $C_G(x) = \langle x \rangle \times \Alt_{n-m}$ where we identify $\Alt_{n-m}$ with the alternating group on the set $[n] \setminus [m]$.
    
    Let $y$ be as in Equation \eqref{eq:cycle_y}. We want to show that $yC_G(x) \cap \calC = \{ y \}$. To this end, let $g \in C_G(x)$ be such that $yg \in \calC$. By the same argument as in the proof of \autoref{pro:5_cycles_satisfy_2_element_condition}, $g$ has to act trivially on $[n] \setminus [m]$.

    First, suppose that $g$ restricts on $[m]$ to $x^{-\gamma} = x^{m-\gamma}$ for some $6 \leq \gamma \leq m-1$. Pick $a,b \in [m]$ with $1 + \gamma \leq a,b \leq \frac{m-5}{2} + \gamma$ and $a+b = m+1+\gamma$. This is possible since $2(1+\gamma) \leq m+1+\gamma$ holds for $\gamma \leq m-1$ and $2(\frac{m-5}{2} + \gamma) \geq m+1+\gamma$ is valid for $\gamma \geq 6$. With these $a,b$, we observe that
    \[
    (yg)(a) = (\pi_m x^{-\gamma})(a) = \pi_m(a - \gamma) = m+1-(a-\gamma) = m+1+\gamma - a = b,
    \]
    and by a similar calculation that $(yg)(b) = a$. Therefore, $yg$ either has a fixed point (if $a=b)$ or a $2$-cycle $(a \ b)$ (if $a \neq b)$ on $[m]$.

    Now suppose that $g$ restricts on $[m]$ to $x^{\gamma}$ for some $5 \leq \gamma \leq m-2$. Pick $a,b \in [m]$ with $\frac{m+7}{2}- \gamma \leq a,b \leq m-\gamma$ and $a+b = m+2-\gamma$. This is possible since $2(m -\gamma) \geq m+2-\gamma$ holds for $\gamma \leq m-2$ and $2(\frac{m+7}{2}- \gamma) \leq m+2-\gamma$ is valid for $\gamma \geq 5$. With this choice of $a,b$, we obtain
    \[
    (yg)(a) = (\pi_m x^{\gamma})(a) = \pi_m(a + \gamma) = m+2-(a+\gamma) = m+2-\gamma - a = b,
    \]
    and by a similar calculation that $(yg)(b) = a$. It follows again that $yg$ has a fixed point or a $2$-cycle on $[m]$.

    If $g$ restricts on $[m]$ to $x^{m-1} = x^{-1}$, we calculate
    \[
    (yg) \Big(\frac{m-1}{2} \Big) = (\pi_m x^{m-1}) \Big(\frac{m-1}{2}\Big) = \pi_m\Big(\frac{m-3}{2}\Big) = \frac{m-1}{2},
    \]
    so again, $yg$ has a fixed point.

    By what we have shown, $yg \not\in \calC$ whenever $g$ restricts on $[m]$ to $x^{\gamma}$ with $1 \leq \gamma \leq m-6$ or $5 \leq \gamma \leq m-1$.

    As $m \geq 9$, this leaves us with the case $m = 9$, $\gamma = 4$ which is tackled by observing that $(\pi_9 x^4)(1) = 1$.

    We conclude that $yg \in \calC$ if and only if $g$ acts trivially on $[m]$. Therefore, $\calC$ satisfies \eqref{eq:1EC_cent}.
\end{proof}

It is easily verified that Propositions \ref{pro:3cycles_satisfy_2EC} and \ref{pro:5_cycles_satisfy_2_element_condition} also hold true in $\Sym_n$ for $n \geq 5$ and Propositions \ref{pro:7_cycles_satisfy_1_element_condition} and \ref{pro:odd_cycles_satisfy_1EC} in $\Sym_n$ for $n \geq 9$.
Thus in order to prove Theorem \ref{thm:alt_cycles_summary} for $\Sym_n$, one only needs to verify that $K_\calC$ is non-degenerate for $\calC$ a conjugacy class of odd cycles in $\Sym_n$ that splits into two non-real classes in $\Alt_n$.

\begin{proposition}
    Let $m,n$ be integers with $n \geq 5$, $m = 2k+1$ for some $k \geq 1$ odd and $m \leq n \leq m+1$. 
    Let $\calC \subset \Sym_n$ be the conjugacy class of the permutation $x = (1 \ 2 \ \dots \ m)$. 
    Then $\calC$ satisfies \eqref{eq:1EC_cent}. In particular, $K_\calC$ is non-degenerate.
\end{proposition}

\begin{proof}
    The cases $k=1$ and $k=3$ have already been covered in Propositions \ref{pro:3cycles_satisfy_2EC} and \ref{pro:7_cycles_satisfy_1_element_condition} respectively. Hence we can assume $k \geq 5$ and thus $m \geq 11$.
    
    We can again take $y$ as in Equation \eqref{eq:cycle_y}. Then $x$ and $y$ are trivially conjugate as there is only one conjugacy class of $m$-cycles in $\Sym_n$. It now follows that $x$ and $y$ satisfy \eqref{eq:1EC_cent} in the exact same way as in the proof of Proposition \ref{pro:odd_cycles_satisfy_1EC}.
\end{proof}

\section{Computational evidence for \autoref{conj:non-degeneracy}} 
\label{sec:computations}

In this final section, we briefly explain how to check computationally that \autoref{conj:non-degeneracy} holds for all finite simple groups up to order $10^9$. This substantially extends the computational results obtained in \cite{LopezPenaMajidRietsch}.

Recall that conditions \eqref{eq:1EC_conj} and \eqref{eq:2EC_conj} ensure that \autoref{conj:non-degeneracy} holds true for a real conjugacy class $\calC \subseteq G$ for a group $G$.

As $\calC_{{}^gy,{}^gx} = {}^g\calC_{y,x}$ holds for all $g \in G$ and $x,y \in \calC$ (\autoref{lem:shifting_Cyx}), one can fix an element $x = x_0 \in \calC$ in both conditions and ask for the mere existence of an element $y$ such that the conditions are satisfied. Therefore, checking \eqref{eq:1EC_conj} and \eqref{eq:2EC_conj} mostly reduces to the computation of ${}^yx_0$ for all $y \in \calC$ and a fixed $x_0 \in \calC$.

By \autoref{thm:PSL2_all_nondeg}, one can skip the computations for the groups $\PSL_2(q)$ which make up most of the small non-abelian simple groups. Using a \texttt{GAP} script \cite{GAP4}, we have checked \eqref{eq:1EC_conj} and, if necessary, \eqref{eq:2EC_conj} for all real conjugacy classes of finite simple groups $G$ of order $|G| \leq 10^9$. In order to confirm \eqref{eq:1EC_conj} for a real conjugacy class $\calC \subseteq G$, one fixes an element $x_0 \in \calC$, calculates the elements ${}^yx_0$, where $y$ ranges over $\calC$, and checks for elements that have been calculated exactly once. This can be done in linear time using a dictionary. Note that the \texttt{atlasrep}-package (\cite{atlasrep}) provides permutation representations of finite simple groups, therefore sorting lists of their elements is computationally easy.

If \eqref{eq:1EC_conj} does not hold, one checks in the same way for elements that have been calculated exactly twice, say ${}^yx_0 = {}^{y'}x_0$ ($y \neq y'$), and searches for an element $h \in G$ of odd order such that ${}^hy = y'$.

Note that it is easily seen that applying an automorphism of $G$ to a conjugacy class results in a conjugacy class with an equivalent Killing form, therefore it is sufficient to compute the normalizer $N$ of the permutation representation of $G \leq \Sym_n$ and check only one conjugacy class in each $N$-orbit.

The supplementary \texttt{GAP}-script is available upon request from the authors of this article.

We summarize the findings of the computations. Recall that \eqref{eq:1EC_conj} is a special case of \eqref{eq:2EC_conj}, so we will not mention it separately in the following.

\begin{theorem}
    Each real conjugacy class in a non-abelian simple group $G$ with $|G| \leq 10^9$ satisfies \eqref{eq:2EC_conj}. In particular, each such conjugacy class has a non-degenerate Killing form.
\end{theorem}

In view of this theorem, it is reasonable to ask the following question.

\begin{question}
    Does each real conjugacy class in a non-abelian simple group satisfy \eqref{eq:2EC_conj}?
\end{question}

We reran the code for all, not necessarily real, conjugacy classes in non-abelian simple groups of order up to $10^8$ and found that in each case, \eqref{eq:2EC_conj} is satisfied. This motivates the following, more general question.

\begin{question}
    Does each conjugacy class in a non-abelian simple group satisfy \eqref{eq:2EC_conj}?
\end{question}

It might also be worthwhile to address these questions for further classes of finite groups that are not necessarily simple, such as almost simple or quasi-simple groups.

\end{document}